\documentclass{svjour3}                     
\usepackage{graphicx}
\usepackage{amssymb}
\usepackage{amsmath}
\usepackage{algorithm} 
\usepackage{algorithmic}
\usepackage[misc]{ifsym}
\usepackage{multirow}
\usepackage{algorithm,float}

\usepackage{xcolor}        
\usepackage{longtable}
\usepackage{booktabs}    
\usepackage{tabularx}    
\usepackage{multirow}    
\usepackage{makecell}    
\usepackage{siunitx}     
\usepackage{caption}
\usepackage{longtable}
\usepackage{array}
\usepackage{needspace}

\usepackage{tabularray}
\usepackage{manyfoot}%
\usepackage{url}

\makeatletter
\newenvironment{breakablealgorithm}
  {
   \begin{center}
     \refstepcounter{algorithm}
     \hrule height.8pt depth0pt \kern2pt
     \renewcommand{\caption}[2][\relax]{
       {\raggedright\textbf{\ALG@name~\thealgorithm} ##2\par}%
       \ifx\relax##1\relax 
         \addcontentsline{loa}{algorithm}{\protect\numberline{\thealgorithm}##2}%
       \else 
         \addcontentsline{loa}{algorithm}{\protect\numberline{\thealgorithm}##1}%
       \fi
       \kern2pt\hrule\kern2pt
     }
  }{
     \kern2pt\hrule\relax
   \end{center}
  }
\makeatother

\renewcommand{\thesection}{\arabic{section}.}

\renewcommand{\thesubsection}{\arabic{section}.\arabic{subsection}}

\newtheorem{them}{Theorem}[section]
\newtheorem{lem}{Lemma}[section]

\newtheorem{pro}{Proposition}[section]
\newtheorem{rem}{Remark}[section]

\begin{document}

\title{New Outer Approximation Algorithms for  Nonsmooth Convex MINLP Problems}

\titlerunning{New OA Algorithm for Convex MINLP}



\author{Zhou Wei \and He-Yi Liu \and Bo Zeng}


\institute{Zhou Wei \at Hebei Key Laboratory of Machine Learning and Computational Intelligence \& College of Mathematics and Information Science, Hebei University, Baoding, 071002, China\\ \email{weizhou@hbu.edu.cn} \\ 
He-Yi Liu \at College of Mathematics and Information Science, Hebei University, Baoding, 071002, China \\ \email{hyliu@stumail.hbu.edu.cn} \\
Bo Zeng  \at Department of Industrial Engineering, University of Pittsburgh, Pittsburgh, PA 15261, USA\\ \email{bzeng@pitt.edu}
}


\date{Received: date / Accepted: date}

\maketitle

\begin{abstract}
This paper presents a novel outer approximation algorithm for nonsmooth mixed-integer nonlinear programming (MINLP) problems. The method proceeds by fixing the integer variables and solving the resulting nonlinear convex subproblem. When the subproblem is feasible, valid linear cuts are derived by computing suitable subgradients of the objective and constraint functions at the optimal solution, utilizing KKT optimality conditions. A new parameter, defined through the nonlinear constraint functions, is introduced to facilitate the generation of these cuts. For infeasible subproblems, a feasibility problem is solved, and valid linear cuts are generated via KKT-based subgradients to exclude the infeasible integer assignment.

By integrating both types of cuts, a mixed-integer linear programming (MILP) master problem is formulated and proven equivalent to the original MINLP. This equivalence underpins a new outer approximation algorithm, which is guaranteed to terminate after a finite number of iterations.

Numerical experiments on smooth convex MINLP problems demonstrate that the proposed algorithm produces tighter MILP relaxations than the classical outer approximation method. Furthermore, the approach offers an alternative mechanism for generating linear cuts, extending beyond reliance solely on first-order Taylor expansions and shows that the efficiency of outer approximation algorithm is strongly dependent on the inherent structure of the MINLP problem.


\end{abstract}

\section{Introduction}


The primary objective of this work is to develop a novel outer-approximation method for nonsmooth convex MINLP. This involves to reformulating a convex MINLP into an equivalent MILP master problem, construct an algorithmic framework to solve a sequence of resulting MILP relaxations, guaranteeing convergence to the optimal solution, and finally prove finite termination of the constructed algorithm.

\subsection{MINLP Problems and Solution Methods}


Many optimization problems encountered in engineering and science involve both discrete and continuous decision variables within nonlinear systems, which are formally described as mixed-integer nonlinear programming (MINLP) problems.
Such MINLP can be represented mathematically as follows:

\begin{equation}\label{1.1}
	(P) \left \{
	\begin{aligned}
		\min\limits_{x,y} \   & f(x, y)\\
		{\rm s.t.}\ \  & g_i(x, y)\leq 0, i=1,\cdots,m,\\
		& x\in X, y\in Y\cap \mathbb{Z}^p.
	\end{aligned}
	\right.
\end{equation}
where $f,g_i: \mathbb{R}^n\times \mathbb{R}^p\rightarrow \mathbb{R}$ $ (i=1,\cdots,m)$ are nonlinear functions, $X\subseteq\mathbb{R}^n$ is a nonempty bounded polyhedral set and $Y\subseteq\mathbb{R}^p$ is a polyhedral set.

MINLP offers a powerful framework for addressing a wide range of real-world optimization challenges. Its applications extend across numerous fields, including engineering (e.g., block layout, process synthesis, water network design), finance (portfolio optimization), energy (nuclear reactor core management), healthcare (cancer treatment planning), and industrial operations (production planning and petrochemical pooling problems). We refer the reader to bibliographies \cite{Biegler2004,Bonami2009,Boukouvala2016,Bragalli2012,Cao2011,Castillo2005,Misener2009advances,Quist1999} and references therein for more details on MINLP applications.

It is widely established that the majority of deterministic approaches for solving convex MINLP problems rely on decomposition strategies. In these methods, the optimal solution is attained by iteratively solving a series of more tractable subproblems. Key deterministic techniques include the extended cutting plane (ECP) method, the extended supporting hyperplane (ESH) method, generalized Benders decomposition (GBD), outer approximation (OA) and LP/NLP-based branch-and-bound. The ECP algorithm, introduced by Westerlund and Petterson \cite{WesterlundPettersson1995}, extends Kelly's cutting plane approach for convex NLP problems developed in \cite{Kelley1960}. Subsequently, Kronqvist et al. \cite{KronqvistLundellWesterlund2016} proposed the ESH method as a specialized algorithm for convex MINLP. Generalized Benders decomposition, first formulated by Geoffrion \cite{Geoffrion1972}, generalizes the earlier Benders decomposition method in \cite{Benders1962} for mixed-integer linear programming (MILP). In GBD, the master problem is projected onto the integer variable space, thereby being expressed solely in terms of discrete variables. The OA method was originally introduced by Duran and Grossmann \cite{DuranGrossmann1986}, with further convex MINLP properties elaborated by Fletcher and Leyffer \cite{FletcherLeyffer1994}. Quesada and Grossmann \cite{QuesadaGrossmann1992} later integrated OA within a branch-and-bound framework based on the successive solution of linear and nonlinear subproblems. This LP/NLP-based branch-and-bound approach maintains a single branch-and-bound tree, in which the MILP master problem is dynamically updated throughout the solution process. In this paper, we focus on a new OA algorithm for solving nonsmooth convex MINLP.

\subsection{Related Research on OA in the Literature}


This subsection provides a review of the evolution of the OA method for solving convex MINLP problems. The pioneering work of Duran and Grossmann \cite{DuranGrossmann1986} in 1986 established the original OA framework for a specific class of MINLP problems where subproblems are affine in the integer variables and convex in the continuous variables.
Fletcher and Leyffer \cite{FletcherLeyffer1994}  in 1994 advanced the method to handle convex MINLPs with continuously differentiable functions, also considering a nonsmooth extension within a penalty function context.
The application of OA was extended to separable nonconvex MINLPs by Kesavan et al. \cite{KesavanAllgorGatzkeBarton2004} in 2004, who proposed algorithms involving sequences of relaxed master problems and nonlinear subproblems.
For nonsmooth convex MINLPs, Eronen and M\"{a}kel\"{a} \cite{Eronen2014}  in 2014 developed a generalized OA approach. This was followed by Wei and Ali \cite{WeiAli2015-JOGO,WeiAli2015-JOTA} in 2015, who further extended OA for nonsmooth problems based on subgradients and KKT optimality conditions.


Recent enhancements have focused on improving convergence and robustness. Su et al. \cite{SuTangBernalGrossmann2018} in 2018 integrated a feasibility pump into the OA scheme, creating a proximal OA method that accelerates the generation of feasible solutions. Delfino and de Oliveira \cite{DelnoDeOliveira2018} combined OA with bundle methods for nonsmooth convex MINLP.



A significant stride was made by Kronqvist Bernal and Grossmann \cite{KronqvistBernalGrossmann2020} in 2020, who efficiently incorporated regularization and second-order derivative information into the OA framework using ideas from the level method. Their use of a second-order Taylor expansion of the Lagrangian significantly reduced iteration counts for highly nonlinear problems.  De Mauri et al. \cite{DeMauriGillisSweversPipeleers2020} derived scaled quadratic cuts based on second-order expansions, providing improved underestimations for convex functions to enhance OA and partial surrogate methods. Coey, Lubin and Vielma \cite{CoeyLubinVielma2020} presented a branch-and-bound LP OA algorithm for convex MINLP reformulated as mixed-integer conic problems, while Muts, Nowak and Hendrix \cite{MutsNowakHendrix2020}  proposed a two-phase decomposition-based OA algorithm that iteratively refines approximations by solving sequences of MILP subproblems.


Building directly on the regularization concepts from \cite{KronqvistBernalGrossmann2020}, Bernal et al. \cite{BernalPengKronqvistGrossmann2022} in 2022 presented a regularized OA method that incorporates new regularization functions based on distance metrics and Lagrangian approximations.

\subsection{Contributions and Outline of the Paper}

In this paper, we conduct an in-depth investigation of the OA method for solving nonsmooth convex MINLP problems. By utilizing solutions derived from feasible subproblems, we introduce new parameters defined in terms of the nonlinear constraint functions and their corresponding solutions. These parameters are then integrated into the OA framework, leading to an MILP master problem that is formally proven to be equivalent to the original MINLP.

The key contributions of our proposed OA method are as follows:
\begin{itemize}
	\item [-] Instead of relying solely on first-order Taylor expansions for linearizing nonlinear constraints, our method offers an alternative linearization approach that enhances flexibility and adaptability.

	\item [-] The proposed linearization scheme generates constraints that define a tighter outer-approximating polyhedron, more closely overestimating the original nonlinear feasible region compared to conventional OA.

	\item [-] Our approach explicitly leverages structural information from the nonlinear constraints in the convex MINLP, demonstrating that problem structure plays a crucial role in solution efficiency, an aspect often underutilized in classical OA implementations.

\end{itemize}
Through these innovations, the proposed OA method improves both theoretical rigor and computational tractability for convex MINLP problems.




The outline of this paper is organized as follows. In Section 2, we give some definitions and preliminaries to be used in this paper.  In Section 3, we develop a new OA method for nonsmooth convex MINLP problems. For feasible subproblems, we construct valid linear cuts using suitable subgradients derived from KKT optimality conditions and enrich them with new parameters defined via the nonlinear constraint functions. For infeasible subproblems, we solve a feasibility problem, select appropriate subgradients, and generate linear cuts to exclude infeasible integers. Integrating both cases, we formulate a MILP master program, which is proven equivalent to the original problem. Based on this equivalence, a new OA algorithm is designed to solve a sequence of MILP relaxations until the optimum of the original problem is attained. Finite termination of the algorithm is established. Example 3.1 illustrates the validity and superior performance of the proposed method compared to the classical OA. Section 4 details the implementation of the new OA algorithm and provides a numerical comparison with the classical OA using test instances from MINLPLib \cite{BussiecDrudMeeraus2003}.
Conclusions of this paper are given in Section 5.

\section{Preliminaries}
Let $\mathbb{R}^n$  be the norm of an Euclidean space equipped with  the inner product  $\langle\cdot, \cdot\rangle$. 
For any subset $C\subset \mathbb{R}^n$, we denote by ${\rm int}(C)$, ${\rm cl}(C)$ and ${\rm conv}(C)$ the interior, the closure and the convex hull of $C$  respectively.

Let $A$ be a closed convex set of $\mathbb{R}^n$ and $x\in A$. We denote by  $T(A, x)$ the contingent cone of $A$ at $x$ that is defined as
$$
T(A, x):={\rm cl}(\mathbb{R}_+(A-x)).
$$
Thus, $v\in T(A, x)$ if and only if there exist $v_k\rightarrow v$ and $t_k\rightarrow 0^+$ such that $x+t_kv_k\in A$ for all $k\in\mathbb{N}$, where $\mathbb{N}$ denotes the set of all natural numbers.

We denote by $N(A, x)$ normal cone of $A$ at $x$ which is defined by
\begin{equation*}
N(A, x):=\{\gamma \in \mathbb{R}^n: \langle \gamma , y-x\rangle\leq 0\ \ {\rm for\ all} \ y\in A\}.
\end{equation*}
It is known that 
\begin{equation*}
  N(A, x)=\{\gamma \in \mathbb{R}^n: \langle \gamma , v\rangle\leq 0\ \ {\rm for\ all} \ v\in T(A, x)\}.
\end{equation*}

Let $\varphi: \mathbb{R}^n\rightarrow \mathbb{R}\cup\{+\infty\}$ be a proper lower semicontinuous convex function and $
x\in {\rm dom}(\varphi):=\{u\in \mathbb{R}^n: \varphi(u)<+\infty\}$. We denote by $\partial\varphi(x)$ the  subdifferential  of $\varphi$ at $x$ that is defined by
\begin{equation}\label{2.1a} 
\partial\varphi(x):=\{\alpha \in \mathbb{R}^n:
\langle \alpha, y-x\rangle\leq\varphi(y)-\varphi(x)\ \ {\rm for\ all} \ y\in \mathbb{R}^n\}
\end{equation}
Each vector in $\partial \varphi(x)$ is said to be a subgradient of $\varphi$ at $x$.

\medskip

The following lemma is cited from \cite[Proposition 2.3]{WeiAli2015-JOGO} and used in our analysis.
\begin{lem}\label{lem2.1}
Let $\phi:\mathbb{R}^n\times\mathbb{R}^p\rightarrow \mathbb{R}$ be a continuous and convex function and suppose that $(\bar x, \bar y)\in \mathbb{R}^n\times\mathbb{R}^p$. Then for any $\alpha\in\partial\phi(\cdot,\bar y)(\bar x)$, there exists $\beta\in\mathbb{R}^p$ such that $(\alpha, \beta)\in\partial\phi(\bar x, \bar y)$.
\end{lem}

The following lemma is taken from \cite[Theorem 3.23]{Phelps1989}.
\begin{lem}\label{lem2.2}
	Let $\varphi_1,\varphi_2: \mathbb{R}^n\rightarrow \mathbb{R}\cup\{+\infty\}$ be proper lower semicontinuous convex functions and $\bar x\in {\rm dom}(\varphi_1)\cap{\rm int}({\rm dom}(\varphi_2))$. Then $$\partial(\varphi_1+\varphi_2)(
	\bar x)=\partial\varphi_1(
	\bar x)+\partial\varphi_2(
	\bar x).$$
\end{lem}

We conclude this section with the following lemma cited from \cite[Theorem 2.4.18]{Zalinescu2002}.
\begin{lem}\label{lem2.3}
Let $\varphi: \mathbb{R}^n\rightarrow \mathbb{R}$ be a continuous convex function. Define
$$\varphi_+(x):=\max\{\varphi(x), 0\}, \ \ \forall x\in\mathbb{R}^n.$$
Then $\varphi_+$ is a convex and continuous function and
\begin{equation*}
  \partial \varphi_+(x)=[0, 1]\partial \varphi(x)
\end{equation*}
holds for all $x\in\mathbb{R}^n$ with $\varphi(x)=0$. 
\end{lem}

\setcounter{equation}{0}

\section{New OA Algorithm for Convex MINLP Problems}

In this section, we aim to derive a new OA algorithm for dealing with nonsmooth convex MINLP problems. To construct this algorithm, we need to solve nonsmooth convex NLP subproblems with detailed computing procedure, including the choice of nonnegative multipliers and subgradients that satisfy the KKT conditions and new parameters defined by these subgradients and nonlinear constraint functions.  

The  nonsmooth convex MINLP problem studied in this section is presented as follows:
\begin{equation}\label{3.1-251030}
	(P) \left \{
	\begin{aligned}
		\min\limits_{x,y} \   & f(x, y)\\
		{\rm s.t.}\ \  & g_i(x, y)\leq 0, i=1,\cdots,m,\\
		   & x\in X, y\in Y\cap \mathbb{Z}^p.
	\end{aligned}
	\right.
\end{equation}
where $f,g_i: \mathbb{R}^n\times \mathbb{R}^p\rightarrow \mathbb{R}$ $ (i=1,\cdots,m)$ are continuous convex functions, $X\subseteq\mathbb{R}^n$ is a nonempty bounded polyhedral set and $Y\subseteq\mathbb{R}^p$ is a polyhedral set.


The main goal of this section is to propose a new outer approximation algorithm for solve convex MINLP problem $(P)$ in \eqref{3.1-251030}. To this aim, it is to use the projection for expressing problem $(P)$ onto $y$ variables. For a fixed $y\in Y\cap \mathbb{Z}^p$, we consider the following nonlinear  subproblem $NLP(y)$:
\begin{equation}\label{3.2-251030}
	NLP(y) \left \{
	\begin{aligned}
		\min\limits_{x} \   & f(x, y)\\
		{\rm s.t.}\ \  & g_i(x, y)\leq 0, i=1,\cdots,m,\\
		& x\in X.
	\end{aligned}
	\right.
\end{equation}
We divide $Y$ into two disjoint subsets:
\begin{equation}\label{3.3-251030}
	T:=\{y\in Y\cap \mathbb{Z}^p: NLP(y) \ {\rm is\ feasible}\} \ \ {\rm and} \ \ S:=\{y\in Y\cap \mathbb{Z}^p: NLP(y) \ {\rm is\ infeasible}\}.
\end{equation}

{\it Throughout this section, we always assume that the following Slater constrain qualification holds:}
\begin{itemize}
	\item[] \rm (A1) {\it For any $y\in T$, 
		there is some $\hat x\in X$ such that $g_i(\hat x, y)<0$ for all $i=1,\cdots,m$.}
\end{itemize}






\subsection{Feasible Nonlinear Subproblem.} Let $y_j\in T$. Then $f(\cdot,y_j), g_i(\cdot,y_j)$ are continuous convex functions with respect to variable $x$.  Note that subproblem $NLP(y_j)$ is a convex optimization problem and then the existence of the optimal solution to $NLP(y_j)$ follows as $X$ is bounded and closed. Thus, we can assume that $x_j\in X$ solves subproblem $NLP(y_j)$. 

We denote by
\begin{equation}\label{3.4a}
I(x_j):=\{i: g_i(x_j,y_j)=0 \}
\end{equation}
the active index set.  

By Assumption (A1) and KKT optimality conditions, 
there exist nonnegative multipliers $\lambda_{j,1},\cdots,\lambda_{j,m}\geq 0$ such that
\begin{equation}\label{3.4-251030}
\left\{	\begin{aligned}
	& 0\in\partial f(\cdot,y_j)(x_j)+\sum\limits_{i\in I(x_j)}^m\lambda_{j,i}\partial g_i(\cdot,y_j)(x_j)+ N(X, x_j),\\
	& \lambda_{j,i}g_i(x_j,y_j)=0, i=1,\cdots,m,
	\end{aligned}
	\right.
\end{equation}
Thus, there exist subgradients $\alpha_j\in\partial f(\cdot,y_j)(x_j)$ and $ \xi_{j,i}\in\partial g_i(\cdot,y_j)(x_j), i=1,\cdots,m$ such that 
\begin{equation}\label{3.5-251030}
  -\alpha_j-\sum\limits_{i=1}^m\lambda_{j,i}\xi_{j,i}\in N(X, x_j).
\end{equation}
By virtue of Lemma \ref{lem2.1}, 
there exist $\beta_j\in\mathbb{R}^p$ and $\eta_{j,i}\in\mathbb{R}^p, i=1,\cdots,m$ such that
\begin{equation}\label{3.6-251030}
	(\alpha_j, \beta_j)\in\partial f(x_j,y_j) \ {\rm and} \ (\xi_{j,i},\eta_{j,i})\in\partial g_i(x_j,y_j), i=1,\cdots,m.
\end{equation}

The following proposition is useful for the new OA method to equivalently reformulate convex MINLP problem $(P)$.

\begin{pro}\label{pro3.1}
	Let $x_j\in X$ solve feasible subproblem $NLP(y_j)$ and subgradient $\alpha_j, \xi_{j,i}, i=1,\cdots,m$ satisfy \eqref{3.5-251030}. Then for any $r>0$, one has
	\begin{equation}\label{3.7-251030}
		\alpha_j^T(x-x_j)\geq 0
	\end{equation}
holds for all $x\in X$ with $g_i(x_j, y_j)+r \xi_{j,i}^T(x-x_j)\leq 0, i=1,\cdots, m$.
\end{pro}

{\bf Proof.} Let $r>0$ and $x\in X$ such that
\begin{equation}\label{3.8-251030}
	g_i(x_j, y_j)+r \xi_{j,i}^T(x-x_j)\leq 0, i=1,\cdots, m.
\end{equation}
Then $x-x_j\in T(X, x_j)$ and it follows from \eqref{3.5-251030} that 
\begin{equation}\label{3.9-251030}
\Big(\alpha_j+\sum\limits_{i\in I(x_j)}\lambda_{j,i}\xi_{j,i}\Big)^T(x-x_j)\geq 0.
\end{equation}
Multiply by $\lambda_{j,i}$ with the left side in \eqref{3.8-251030}, \eqref{3.4-251030} gives that 
$$
\sum_{i=1}^m r \lambda_{j,i}\xi_{j,i}^T(x-x_j)\leq 0
$$
and thus \eqref{3.7-251030} holds by \eqref{3.9-251030}. The proof is complete. \hfill$\Box$

\medskip

The following proposition follows immediately from Proposition \ref{pro3.1}.

\begin{pro}\label{pro3.2}
	Let $x_j\in X$ solve feasible subproblem $NLP(y_j)$ and subgradients $\alpha_j,\xi_{j,i}(i=1,\cdots,m)$ satisfy \eqref{3.5-251030}.
	Let $r>0$ be fixed. Consider the following linear programming $LP(r,y_j)$:
	
\begin{equation}\label{3.10-251030}
	LP(r,y_j) \left \{
\begin{aligned}
		\min_{x} \ & f(x_j, y_j)+ \alpha_j^T(x-x_j)\\
		{\rm s.t.}\   & g_i(x_j, y_j)+ r\xi_{j,i}^T(x-x_j)\leq 0, i=1,\cdots,m,\\
		&  x\in X.
\end{aligned}
	\right.
\end{equation}
Then $x_j$ is an optimal solution to $LP(r,y_j)$ with  the optimal value $f(x_j,y_j)$.
\end{pro}

Different from the classic OA for dealing with convex MINLP problems, we introduce a new parameter $\rho_j$ for the feasible nonlinear subproblem case and it would be proved to produce a valid linear cut and a better solution when solving MILP relaxations (see Example 3.1).

 \medskip
 
We denote by
\begin{equation}\label{3.21-251030}
	J(x_j):=\{1,\cdots,m\}\backslash I(x_j).
\end{equation}
If $J(x_j)\not=\emptyset$, we set
\begin{equation}\label{3.23-251030}
	\Pi_j:=\max_{i\in J(x_j)}\max_{(x,y)\in X\times (Y\cap \mathbb{Z}^p)}\left\{(\xi_{j,i}, \eta_{j,i})^T\begin{pmatrix}x-x_j\\y-y_j\end{pmatrix}\right\}.
\end{equation}
We select a parameter $\rho_j>0$ as follows:
\begin{equation}\label{3.24-251030}
	\rho_j:=\left \{
	\begin{aligned}
		\frac{-\max\limits_{i\in J(x_j)}g_i(x_j,y_j)}{\Pi_j}, \ &  {\rm if} \ J(x_j)\not=\emptyset, \Pi_j> 0,\\
		1,\ \ \ \ \ \ \ \ \ \ \ \ \ \ \ & {\rm otherwise}.
	\end{aligned}
	\right.
\end{equation}

\begin{rem}

	It is noted that  $\Pi_j$ refers to computing fintely many MILP problems. Further, for each  $i\in J(x_j)\not=\emptyset$, we consider the following problem:
\begin{equation}\label{3.15-260126}
	{\rm P}_i\left\{\begin{aligned}
		\max \  &(\xi_{j,i}, \eta_{j,i})^T\begin{pmatrix}x-x_j\\y-y_j\end{pmatrix}\\
		{\rm s.t.}\ & x\in X, y\in {\rm conv}(Y\cap \mathbb{Z}^p).
	\end{aligned}\right.
\end{equation}
We denote by $\Upsilon_i$ the optimal value of ${\rm P}_i$. Then it is easy to verify that $$\Pi_j=\max_{i\in J(x_j)}\Upsilon_i.$$ 
Each problem ${\rm P}_i$ refers to the convex hull of $Y\cap \mathbb{Z}^p$ and it is impractical to compute ${\rm conv}(Y\cap \mathbb{Z}^p)$ in general. However, for some cases with the speical structure (i.e. $Y$ is the integral polyhedron), it may be easy to get $\Pi_j$ by solving finitely many problems ${\rm P}_i$ given in \eqref{3.15-260126}.


\end{rem}

To equivalently reformulate the convex MINLP problem in \eqref{3.1-251030}, we consider the following MILP problem:

\begin{equation}\label{3.25-251030}
	\left \{
	\begin{aligned}
		\min_{x,y,\theta} \ \  &\theta\\
		{\rm s.t.}\ \ & f(x_j, y_j)+(\alpha_j, \beta_j)^T\begin{pmatrix}x-x_j\\y-y_j\end{pmatrix}\leq \theta\ \ \ \forall y_j\in T, \\
		\ & g_i(x_j, y_j)+\rho_j(\xi_{j,i}, \eta_{j,i})^T\begin{pmatrix}x-x_j\\y-y_j\end{pmatrix}\leq 0\ \ \forall i,\ \forall y_j\in T, \\
		\ \    &x\in X, y\in T, \theta\in\mathbb{R}.
	\end{aligned}
	\right.
\end{equation}


The following theorem is to prove the equivalence of the convex MINLP problem  $(P)$ given in \eqref{3.1-251030} and MILP presented in \eqref{3.25-251030}. This result is crucial as it provides the theoretical guarantee for the construction of the new OA algorithm designed to solve problem $(P)$.

\begin{them}\label{th3.1}
	The MILP problem of \eqref{3.25-251030} is equivalent to convex MINLP problem $(P)$ of \eqref{3.1-251030} in the sense that both problems have the same optimal value and the optimal solution $(\bar x, \bar y)$ to problem $(P)$ corresponds to the optimal solution $(\bar x, \bar y, \bar\theta)$ to problem of \eqref{3.25-251030} with $\bar\theta=f(\bar x, \bar y)$.
\end{them}

{\bf Proof.} Suppose that $(\bar x,\bar y,\bar\theta)$ is an optimal solution  MILP problem of \eqref{3.25-251030} and $(x_{j_0},y_{j_0})$ solves the convex MINLP problem (P) of \eqref{3.1-251030}. We need to prove that $\bar\theta=f(x_{j_0},y_{j_0})$.

We assume that $\bar y=y_{j_1}$ for some $y_{j_1}\in T$. Then
\begin{equation}\label{3.26}
	f(x_{j_1}, y_{j_1})+(\alpha_{j_1},\beta_{j_1})^T\begin{pmatrix}\bar x-x_{j_1}\\ 0\end{pmatrix}\leq \bar\theta
\end{equation}
and
$$
g_i(x_{j_1}, y_{j_1})+\rho_j(\xi_{j,i}, \eta_{j,i})^T\begin{pmatrix}\bar x-x_{j_1}\\0\end{pmatrix}\leq 0, i=1,\cdots,m.
$$
This and Proposition \ref{pro3.1} imply that $\alpha_{j_1}^T(\bar x-x_{j_1})\geq 0$ and by \eqref{3.26}, one has
$$
\bar\theta\geq f(x_{j_1}, y_{j_1})\geq f(x_{j_0}, y_{j_0})
$$
(the second inequality follows as $(x_{j_0}, y_{j_0})$ is the optimal solution to $P$).

We next prove that $\bar\theta\leq f(x_{j_0}, y_{j_0})$. To this aim, it suffices to show that $(x_{j_0},y_{j_0})$ satisfies all constraints with respect to $g_i$ in \eqref{3.25-251030}. Granting this, it follows from the convexity of $f$ that $(x_{j_0},y_{j_0}, f(x_{j_0}, y_{j_0}))$ is feasible to problem of \eqref{3.25-251030}  and consequently $\bar\theta\leq f(x_{j_0}, y_{j_0})$.

Let $y_j\in T$. We only need to consider the case $\rho_j>0$. We first claim that
\begin{equation}\label{3.27}
	g_i(x_j, y_j)+\rho_j(\xi_{j,i}, \eta_{j,i})^T\begin{pmatrix}x_{j_0}-x_j\\y_{j_0}-y_j\end{pmatrix}\leq 0, \ \forall i\in I(x_j).
\end{equation}

Suppose on the contrary that there exists $i\in I(x_j)$ such that
$$
g_i(x_j, y_j)+\rho_j(\xi_{j,i}, \eta_{j,i})^T\begin{pmatrix}x_{j_0}-x_j\\y_{j_0}-y_j\end{pmatrix}>0
$$
and consequently
$$
(\xi_{j,i}, \eta_{j,i})^T\begin{pmatrix}x_{j_0}-x_j\\y_{j_0}-y_j\end{pmatrix}>0
$$
(thanks to $i\in I(x_j)$ and $\rho_j>0$). Hence the convexity of $g_i$ gives that $$g_i(x_{j_0},y_{j_0})\geq g_i(x_j,y_j)+(\xi_{j,i}, \eta_{j,i})^T\begin{pmatrix}x_{j_0}-x_j\\y_{j_0}-y_j\end{pmatrix}=(\xi_{j,i}, \eta_{j,i})^T\begin{pmatrix}x_{j_0}-x_j\\y_{j_0}-y_j\end{pmatrix}>0,$$ which contradicts $g_i(x_{j_0},y_{j_0})\leq 0$. This means that \eqref{3.27} holds.

Let $k\in J(x_j)$. By the choice of $\rho_j$ in \eqref{3.24-251030}, one has
\begin{eqnarray*}
	&&g_k(x_j, y_j)+\rho_j(\xi_{j,k}, \eta_{j,k})^T\begin{pmatrix}x_{j_0}-x_j\\y_{j_0}-y_j\end{pmatrix}\\
	&=& g_k(x_j, y_j)+\frac{-\max\limits_{i\in J(x_j)}g_i(x_j,y_j)}{\Pi_j}(\xi_{j,k}, \eta_{j,k})^T\begin{pmatrix}x_{j_0}-x_j\\y_{j_0}-y_j\end{pmatrix}\\
	&\leq& g_k(x_j, y_j)+\Big(-\max\limits_{i\in J(x_j)}g_i(x_j,y_j)\Big)\\
	&\leq& 0
\end{eqnarray*}
This and \eqref{3.27} imply that $(x_{j_0},y_{j_0})$ satisfies all constraints with respect to $g_i$ in \eqref{3.25-251030}. Hence $\bar\theta=f(x_{j_0},y_{j_0})$.  The proof is complete. \hfill$\Box$ 


\begin{rem}
Our method departs from the classical OA approach for convex MINLPs by introducing parameters $\rho_j$ in the treatment of feasible nonlinear subproblems. Theorem 3.1 proves the validity of the resulting linear cut and shows that its incorporation yields a tighter polyhedral approximation of the nonlinear feasible region within the MILP relaxation. This tightening effect is empirically demonstrated in Example 3.1.
\end{rem}


\subsection{Infeasible Nonlinear Subproblem.} Let  $y_l\in S$. Then subproblem $NLP(y_l)$ is infeasible. For the construction of the algorithm, it is necessary to generate the valid linear cut so as to exclude such infeasible integer $y_l$. To this aim, we consider the following nonlinear feasibility problem $F(y_l)$:
\begin{equation}\label{3.11-251031}
	F(y_l)\left \{
	\begin{aligned}
		\min_{x,u} \  &\sum\limits_{i=1}^mu_i\\ 
		{\rm s.t.} \  &g_i(x, y_l)\leq u_i,i=1,\cdots,m\\
		  & u_i\geq 0,i=1,\cdots,\\
		  & x\in X, u\in\mathbb{R}^m.
	\end{aligned}
	\right.
\end{equation}

For any $i$, we denote by $\varphi_i(x):=\max\{g_i(x, y_l), 0\}$ for each $x$. Note that $X$ is bounded and closed, and thus we can assume that $x_l\in X$ solves problem $F(y_l)$. Then one can verify that $x_l$ is also the optimal solution to 
the following nonlinear problem $F(y_l)^{\prime}$:
\begin{equation*}\label{3.11-251030}
	F(y_l)^{\prime}\left \{
	\begin{aligned}
		\min_{x} \  &\sum_{i=1}^m \varphi_i(x)\\ 
		{\rm s.t.} \   & x\in X.
	\end{aligned}
	\right.
\end{equation*} 
Then by using KKT optimality conditions and Lemma \ref{lem2.2}, one has
\begin{equation}\label{3.12-251030}
		0\in\partial\Big(\sum_{i=1}^m\varphi_i\Big)(x_l)+N(X,x_l)=\sum_{i=1}^m\partial\varphi_i( x_l)+N(X,x_l).
\end{equation}
We divide $\{1,\cdots,m\}$ into three disjoint subsets 
\begin{equation}\label{3.13-251030}
\left\{	\begin{aligned}
	I(x_l):=&\{i\in \{1,\cdots,m\}: g_i(x_l,y_l)=0 \},\\
	I^<(x_l):=&\{i \in \{1,\cdots,m\}: g_i(x_l,y_l)<0\},  \\	I^>(x_l):=&\{i \in \{1,\cdots,m\}: g_i(x_l,y_l)>0\}.
	\end{aligned}
	\right.
\end{equation}
Using the continuity of $g_i$ and Lemma \ref{lem2.3}, one can verify that 
\begin{equation}\label{3.15-251031}
\partial\varphi_i( x_l)=\left\{
	\begin{aligned}
		\{0\},\ \ \ \ \ \ \ \  &i\in I^<(x_l),\\
		[0,1]\partial g_i(\cdot, y_l)(x_l),\ &i\in I(x_l),\\
		\partial g_i(\cdot, y_l)(x_l), \ \ \ & i\in I^>(x_l).
	\end{aligned}
	\right.
\end{equation}
Set
\begin{equation}\label{3.16-201031}
\lambda_{l,i}=0,\ {\rm if} \ i\in I^<(x_l) \  {\rm and} \ 
\lambda_{l,i}=1,\ {\rm if} \ i\in I^>(x_l).
\end{equation}
Then by virtue of \eqref{3.15-251031}, there exist nonnegative multipliers $\lambda_{l,i}\in [0, 1], i\in I(x_l)$ and $\xi_{l,i}\in \partial g_i(\cdot, y_l)(x_l),i=1,\cdots,m$ such that
\begin{equation}\label{3.16-251030}
	-\sum\limits_{i=1}^m\lambda_{l,i}\xi_{l,i}\in N(X,x_l).
\end{equation}
Using Lemma \ref{lem2.1}, there are $\xi_{l,i}\in\partial g_i(\cdot,y_l)(x_l)$, $\eta_{l,i}\in\mathbb{R}^p, i=1,\cdots,m$ such that
\begin{equation}\label{3.17-251030}
(\xi_{l,i}, \eta_{l,i})\in\partial g_i(x_l,y_l), i=1,\cdots,m.
\end{equation}


The following proposition shows that one can generate a valid linear cut by subgradients $(\xi_{l,i}, \eta_{l,i})$ to exclude infeasible integer $y_l$. 
\begin{pro}\label{pro3.3}
	Let $x_l\in X$ solve nonlinear problem $F(y_l)^{\prime}$ and take nonnegative multipliers $\lambda_{l,i}$ and subgradients $(\xi_{l,i}, \eta_{l,i}), i=1,\cdots, m$ satisfying \eqref{3.16-201031}, \eqref{3.16-251030} and \eqref{3.17-251030}. Then the following constraints:
\begin{equation}\label{3.8a}
\left \{
\begin{array}l
g_i(x_l, y_l)+(\xi_{l,i}, \eta_{l,i})^T\begin{pmatrix}x-x_l\\y-y_l\end{pmatrix}\leq 0, i=1,\cdots,m,\\
x\in X, y\in Y\cap \mathbb{Z}^p.
 \end{array}
\right.
\end{equation}
could exclude the infeasible integer variable $y_l$.
\end{pro}

{\bf Proof.} Note that subproblem $NLP(y_l)$ is infeasible and thus the optimal value of $F(y_l)$ is positive; that is, 
\begin{equation*}\label{3.19-251031}
\sum_{i=1}^m \varphi_i(x_l)>0.
\end{equation*}
Suppose on the contrary that there exists $\hat x\in X$ such that $(\hat x, y_l)$ is feasible to the constraint of \eqref{3.8a}. Then
\begin{equation}\label{3.19-251030}
	g_i(x_l, y_l)+ \xi_{l,i}^T(\hat x-x_l)\leq 0,
	\ \ \forall i=1,\cdots,m.
\end{equation}
By virtue of \eqref{3.16-251030}, one has
\begin{equation*}
	\sum_{i=1}^m \lambda_{l,i}\xi_{l,i}^T(\hat x-x_l)\geq 0
\end{equation*}
Multiplying by $\lambda_{l,i}$ with \eqref{3.19-251030}, one has
\begin{eqnarray*}
	0 &\geq & 	\sum_{i=1}^m\lambda_{l,i}g_i(x_l, y_l)+\sum_{i=1}^m \lambda_{l,i}\xi_{l,i}^T(\hat x-x_l)\\
	&=& \sum_{i\in I^>(y_l)}\lambda_{l,i}g_i(x_l, y_l)+\sum_{i=1}^m \lambda_{l,i}\xi_{l,i}^T(\hat x-x_l)\\
	&\geq& \sum_{i\in I^>(y_l)}\lambda_{l,i}g_i(x_l, y_l)\\
	&=&\sum_{i=1}^m \varphi_i(x_l)>0
\end{eqnarray*}
(the last equality holds by the choice of $\lambda_{l,i}$), which is a contradiction. The proof is complete. \hfill $\Box$


\subsection{Construction of New OA Algorithm}
To construct the new OA algorithm for solving convex MINLP problem $(P)$ in \eqref{3.1-251030}, it is necessary to  equivalently reformulate the  problem $(P)$, and thus we need to combine feasible with infeasible nonlinear subproblems. Based on Theorem \ref{th3.1} and Proposition \ref{pro3.3}, we have the following theorem that is to reformulate convex MINLP problem $(P)$ in \eqref{3.1-251030} as an equivalent MILP  master program. This result is to guarantees in theory the construction of the corresponding new OA algorithm for solve problem $(P)$.

\begin{them}\label{th3.2}
	For any $y_j\in T$,
let $x_j\in X$ solve feasible subproblem $NLP(y_j)$ and subgradient $(\alpha_j,\beta_j), (\xi_{j,i},\eta_{j,i}), i=1,\cdots,m$ satisfy \eqref{3.5-251030} and \eqref{3.6-251030} and select a parameter $\rho_j$ as \eqref{3.24-251030}, and for any $y_l\in S$, let $x_l\in X$ solve nonlinear feasibility problem $F(y_l)$ and subgradients $(\xi_{l,i}, \eta_{l,i}), i=1,\cdots, m$ satisfying  \eqref{3.16-251030} and \eqref{3.17-251030}. Consider the following MILP master program (MP):
	\begin{equation}\label{3.28-251030}
		(MP)\left \{
		\begin{aligned}
			\min_{x,y,\theta} \ &\theta\\
			{\rm s.t.}\ \ & f(x_j, y_j)+(\alpha_j,\beta_j)^T\begin{pmatrix}x-x_j\\y-y_j\end{pmatrix}\leq \theta\ \ \forall y_j\in T, \\
			\ \  & g_i(x_j, y_j)+\rho_j (\xi_{j,i},\eta_{j,i})^T\begin{pmatrix}x-x_j\\y-y_j\end{pmatrix}\leq 0, \forall i, 
			\forall y_j\in T,\\
			\ \  & g_i(x_l, y_l)+(\xi_{l,i}, \eta_{l,i})^T\begin{pmatrix}x-x_l\\y-y_l\end{pmatrix}\leq 0, \forall i, 
			\ \forall y_l\in S,\\
			\ \ &   x\in X, y\in Y\cap \mathbb{Z}^p, \theta\in\mathbb{R}.
		\end{aligned}
		\right.
	\end{equation}
	Then master program $(MP)$ is equivalent to the convex MINLP problem (P) of \eqref{3.1-251030} in the sense that both problems have the same optimal value and that the optimal solution $(\bar x, \bar y)$ to problem (P) corresponds to the optimal solution $(\bar x, \bar y, \bar\eta)$ to (MP) with $\bar\theta=f(\bar x, \bar y)$.
\end{them}

\begin{rem}
The KKT conditions play an important role in the equivalent reformulation of MINLP problem (P), and it is noted that subgradients, if not satisfying KKT conditions, may be insufficient for this equivalent reformulation (cf. \cite[Example 3.1]{WeiAli2015-JOGO}).
\end{rem}




Based on Theorem \ref{th3.2}, we construct a new OA algorithm. The algorithm iteratively solves the relaxed master problem $(MP)$ as defined in \eqref{3.28-251030}, which ensures convergence to the optimal solution of the original problem (P) from \eqref{3.1-251030}.

\medskip
{\it 
At iteration $k$, subsets $T^k$ and $S^k$ are defined as follows:
\begin{equation}\label{3.28-251031}
T^k:=T\cap\{y_1,\cdots, y_k\} \ \ {\rm and} \ \ S^k:=S\cap\{y_1,\cdots, y_k\}.
\end{equation}
Exactly one of the following cases occurs:
\begin{itemize}
	\item[\rm (a)] $y_k\in T^k$:  Solve subproblem $NLP(y_k)$ and denote by $x_k$ the optimal solution. Solve the following optimization problem ${\it OP}(x_k,y_k)$:
\begin{equation}\label{3-10}
{\it Find\ }
\left [
\begin{array}l
  (\alpha_k,\beta_k)\\
  (\xi_{k,i},\eta_{k,i})\\
  \lambda_{k,i}\in\mathbb{R}
 \end{array}
\right ]
{\it such \ that} \  \left\{
\begin{array}l
 -\alpha_k-\sum\limits_{i=1}^m\lambda_{k,i}\xi_{k,i}\in N(X, x_k),\\
 (\alpha_k,\beta_k)\in\partial f(x_k,y_k),\\
 (\xi_{k,i},\eta_{k,i})\in\partial g_i(x_k,y_k), i=1,\cdots,m,\\
 \lambda_{k,i}\geq 0, i=1,\cdots,m.
 \end{array}
 \right.
\end{equation}
Obtain subgradients $(\alpha_k,\beta_k)$ and $(\xi_{k,i},\eta_{k,i})$ for all $i=1,\cdots,m$. Select a parameter $\rho_k>0$ as said in \eqref{3.24-251030}.

\item[\rm (b)] $y_k\in S^k$: Solve $F(y_k)$ and denote the solution by $x_k$. Solve the following optimization problem ${\it Sub\mbox-OP}(x_k,y_k)$:
\begin{equation}\label{3-11}
{\it Find\ }
\left [
\begin{array}l
  (\xi_{k,i},\eta_{k,i})\\
  \lambda_{k,i}\in\mathbb{R}
 \end{array}
\right ]
{\it such \ that} \ \left\{
\begin{array}l
 -\sum\limits_{i=1}^m\lambda_{k,i}\xi_{k,i}\in N(X, x_k),\\
 (\xi_{k,i},\eta_{k,i})\in\partial g_i(x_k,y_k),\  i=1,\cdots,m,\\
  \lambda_{k,i}=0, \ \ \ \ \ {\it if} \ g_i(x_k,y_k)<0,\\
  \lambda_{k,i}=1, \ \ \ \ \  {\it if} \  g_i(x_k,y_k)>0,\\
  \lambda_{k,i}\in [0,1], \  {\it if} \ g_i(x_k,y_k)=0.
 \end{array}
 \right.
\end{equation}
Obtain subgradients $(\xi_{k,i},\eta_{k,i})$ for all $i=1,\cdots,m$.
\end{itemize}
}

Let $UBD^k:=\min \{f(x_j,y_j): y_j\in T^k\}$. Consider the following relaxed master program $MP(T^k, S^k)$:
\begin{equation}\label{4.1a}
MP(T^k, S^k)\left \{
\begin{aligned}
\min_{x,y,\theta}\  &\theta\\
 {\rm s.t.}\ \  &\theta<UBD^k\\
 \ \    &f(x_j, y_j)+(\alpha_j, \beta_j)^T\begin{pmatrix}x-x_j\\y-y_j\end{pmatrix}\leq \theta\ \ \forall y_j\in T^k,\ \  \\
 \ \ & g_i(x_j, y_j)+\rho_j(\xi_{j,i}, \eta_{j,i})^T\begin{pmatrix}x-x_j\\y-y_j\end{pmatrix}\leq 0,\ \forall i, \ \forall y_j\in T^k,\\ 
 \ \ & g_i(x_l, y_l)+(\xi_{l,i}, \eta_{l,i})^T\begin{pmatrix}x-x_l\\y-y_l\end{pmatrix}\leq 0,\ \forall i, \ \forall y_l\in S^k,\\ 
 \ \ &  x\in X, y\in Y\cap \mathbb{Z}^p, \theta\in\mathbb{R}.
\end{aligned}
\right.
\end{equation}

The proposed new OA algorithm for solving the relaxed master problems $MP(T^k,S^k)$ is stated as follow.\\

\begin{breakablealgorithm}
\caption{(New OA Algorithm)}
\begin{algorithmic}[1]
\STATE Initialization. Given an initial $y_0\in Y$, set $T^0=S^0:=\emptyset$, $UBD^0:=\infty$ and let $k:=1$
 \FOR{$k=1,2, \cdots,$}
 \IF {$y_k\in T$}  
 \STATE Solve subproblem $NLP(y_k)$ and denote the solution by $x_k$

 Solve ${\it OP}(x_k,y_k)$ in \eqref{3-10} and obtain subgradients $(\alpha_k,\beta_k),(\xi_{k,i},\eta_{k,i})$

Select a parameter $\rho_j>0$ in \eqref{3.24-251030}.

 Set $T^k:=T^{k-1}\cup\{y_k\}$, $S^k:=S^{k-1}$ and $UBD^k:=\min\{UBD^{k-1}, f(x_k,y_k)\}$

 \ELSE
 \STATE Solve $F(y_k)$ and denote the solution by $x_k$

 Solve ${\it Sub\mbox-OP}(x_k,y_k)$ in \eqref{3-11} and obtain subgradients $(\xi_{k,i},\eta_{k,i})$

 Set $S^k:=S^{k-1}\cup\{y_k\}$, $T^k:=T^{k-1}$ and $UBD^k:=UBD^{k-1}$
\ENDIF

 \STATE Solve the relaxation $MP(T^k,S^k)$ and obtain a new integer $y_{k+1}$

  Set $k:=k+1$ and go back to line 3

\ENDFOR
\end{algorithmic}
\end{breakablealgorithm}

\vspace{0.5cm}
The following theorem is to establish the  convergence of Algorithm 1 in finite time. 
\begin{them}
Suppose that the cardinality of $Y\cap \mathbb{Z}^p$ is finite. Then either problem $(P)$ of \eqref{3.1-251030} is infeasible or Algorithm 1 terminates in a finite number steps at an optimal value of problem $(P)$.
\end{them}

{\bf Proof}. We first prove that any integer variable would not be generated more than once. Granting this, the finite termination of Algorithm 1 follows from the finite cardinality of $Y\cap \mathbb{Z}^p$. 

At iteration $k$, we claim:
 {\it 
 if $(\bar x, \bar y, \bar{\theta})$ solves $MP(T^k,S^k)$, then 
$\bar y\not \in T^k\cup S^k$.}

Let $(\bar x, \bar y, \bar{\theta})$ solve $MP(T^k,S^k)$. Then  Proposition \ref{pro3.3} gives that $\bar y\not\in S^k$ and thus it suffices to prove that $\bar y \not\in T^k$. Suppose on the contrary that $\bar y=y_{j_k}$ for some $y_{j_k}\in T^k$. Then $(\bar x, y_{j_k}, \bar{\theta})$ is feasible to  $MP(T^k,S^k)$ and thus one has
\begin{equation}\label{3.31-251030}
		\bar{\theta}<UBD^k\leq f(x_{j_k}, y_{j_k}),
		f(x_{j_k}, y_{j_k})+\alpha_{j_k}^T( x-x_{j_k})\leq \bar{\theta},
\end{equation}
and 
\begin{equation*}
	g_i(x_{j_k}, y_{j_k})+\rho_{j_k} \xi_{j_k,i}^T (\bar x-x_{j_k})\leq 0, \ i=1,\cdots,m.
\end{equation*}
This and Proposition \ref{pro3.1} imply that
$$
\alpha_{j_k}^T(\bar x-x_{j_k})\geq 0
$$
and consequently \eqref{3.31-251030} gives that $f(x_{j_k}, y_{j_k})\leq \bar{\theta}$, which contradicts $\bar{\theta}<f(x_{j_k}, y_{j_k})$. Hence the claim holds.

Now we can assume that Algorithm 1 terminate at $k$-th step for some $k$. Then the relaxation $MP(T^k, S^k)$ is infeasible. To complete the proof, it remains to prove that the optimal value of problem $(P)$ is attainable.

Let $r^*$ denote the optimal value of problem $(P)$ in \eqref{3.1-251030}. Then $r^*\leq UBD^{k-1}$. We  need to consider two cases:
\begin{itemize}
	\item [(a)] $r^*=UBD^{k-1}$: This implies that $r^*=f(x_j, y_j)$ for some $y_j\in T^{k-1}$. 
	\item[(b)] $r^*<UBD^{k-1}$: Then $y_k\in T^k$ (otherwise, $y_k\in S^k$, $UBD^{k}=UBD^{k-1}$ by Algorithm 1 and consequently $MP(T^k, S^k)$ is feasible, a contradiction). This implies that subproblem $NLP(y_k)$ is feasible and thus $r^*\leq f(x_k, y_k) $. We claim that $r^*= f(x_k, y_k) $ ( Otherwise, $f(x_k, y_k)> r^*$ and thus $r^*<UBD^{k}$, which implies that  $MP(T^k,S^k)$is feasible, a contradiction). 
\end{itemize}
Hence cases (a) and (b) mean that the optimal value of problem $(P)$ can be  attainable.  \hfill$\Box$


\subsection{Compared with Classic OA for Convex MINLP}

In this subsection, we compared the proposed new OA with the classic OA for solving convex MINLP problem in which objective and constraint functions are continuously differentiable.

\medskip

{\it Assume that $f,g_i(i=1,\cdots,m)$ appearing in \eqref{3.1-251030} are continuously differentiable}.

\medskip


The classic OA algorithm for convex MINLPs, such as problem \eqref{3.1-251030}, operates by reformulating the MINLP as an equivalent MILP master program. This is achieved using a first-order Taylor expansion, after which the algorithm solves a sequence of MILP relaxations to converge to the MINLP's optimal solution. In the following section, we detail the key differences between this classical approach and our proposed new OA algorithm.


Since functions  $f,g_i, i=1,\cdots,m$  are differentiable, then one can replace subgradients given 
in Algorithm 1 with gradients of functions $f,g_i$
This leads to a simpler new OA algorithm for solving the convex MINLP in \eqref{3.1-251030} compared to Algorithm 1.

\medskip
{\it  
At iteration $k$, let subsets $T^k$ and $S^k$ are defined as said in \eqref{3.28-251030}.

\begin{itemize}
	\item [(a)] If $y_k\in T^k$, then solve subproblem $NLP(y_k)$ and denote by $x_k$ the optimal solution. 	Set
	\begin{equation*}
		\Pi_k:=\max_{i\in J(x_k)}\max_{(x,y)\in X\times (Y\cap \mathbb{Z}^p)}\left\{\nabla g_i(x_k, y_k)^T\begin{pmatrix}x-x_j\\y-y_j\end{pmatrix}\right\}.
	\end{equation*}	
and  select a parameter $\rho_k>0$ as follows:
	\begin{equation}\label{3.33-251030}
		\rho_k:=\left \{
		\begin{aligned}
			\frac{-\max\limits_{i\in J(x_k)}g_i(x_k,y_k)}{\Pi_k}, \ &  {\rm if} \ J(x_k)\not=\emptyset, \Pi_k\geq 0,\\
			1,\ \ \ \ \ \ \ \ \ \ \ \ \ \ \ & {\rm otherwise}.
		\end{aligned}
		\right.
	\end{equation}
	

	\item [(b)] If $y_k\in S^k$, then solve nonlinear feasibility problem $F(y_k)$ and denote by $x_k$ the optimal solution.
\end{itemize} }

Let $UBD^k:=\min \{f(x_j,y_j): y_j\in T^k\}$. We consider the following MILP relaxation $\widetilde{MP}^{k}$ :
\begin{equation}\label{3.35-251030}
	\widetilde{MP}^k
	\left \{
	\begin{aligned}
		\min_{x,y,\theta}\  &\theta\\
		{\rm s.t.}\ \  &\theta<UBD^k\\
		\ \    &f(x_j, y_j)+\nabla f(x_j, y_j)^T\begin{pmatrix}x-x_j\\y-y_j\end{pmatrix}\leq \theta\ \ \forall y_j\in T^k,\ \  \\
		\ \ & g_i(x_j, y_j)+\rho_j \nabla g_i(x_j, y_j)^T\begin{pmatrix}x-x_j\\y-y_j\end{pmatrix}\leq 0,\ \forall i, \ \forall y_j\in T^k,\\ 
		\ \ & g_i(x_l, y_l)+\nabla g_i(x_l, y_l)^T\begin{pmatrix}x-x_l\\y-y_l\end{pmatrix}\leq 0,\ \forall i, \ \forall y_l\in S^k,\\ 
		\ \ &  x\in X, y\in Y\cap \mathbb{Z}^p, \theta\in\mathbb{R}.
	\end{aligned}
	\right.
\end{equation}

The proposed new OA algorithms for convex MINLP problem of \eqref{3.1-251030} is presented as follows:

\begin{breakablealgorithm}
	\caption{(New OA Algorithm for Convex MINLP)}
	\begin{algorithmic}[1]
		\STATE Initialization. Given an initial $y_0\in Y$, set $T^0=S^0:=\emptyset$, $UBD^0:=\infty$ and let $k:=1$
		\FOR{$k=1,2, \cdots,$}
		\IF {$y_k\in T$}  
		\STATE Solve subproblem $NLP(y_k)$ and denote the solution by $x_k$
		
		Select a parameter $\rho_k>0$ as said in \eqref{3.33-251030}.
		
		Set $T^k:=T^{k-1}\cup\{y_k\}$, $S^k:=S^{k-1}$ and $UBD^k:=\min\{UBD^{k-1}, f(x_k,y_k)\}$
		
		\ELSE
		\STATE Solve $F(y_k)$ and denote the solution by $x_k$
		
		
		Set $S^k:=S^{k-1}\cup\{y_k\}$, $T^k:=T^{k-1}$ and $UBD^k:=UBD^{k-1}$
		\ENDIF
		
		\STATE Solve $\widetilde{MP}^k$ and obtain a new integer $y_{k+1}$
		
		Set $k:=k+1$ and go back to line 3

		\ENDFOR
	\end{algorithmic}
\end{breakablealgorithm}

Let $UBD^k:=\min \{f(x_j,y_j): y_j\in T^k\}$. We consider the following MILP relaxation $MP^k$:
\begin{equation}\label{3.34-251030}
	MP^k\left \{
	\begin{aligned}
		\min_{x,y,\theta}\  &\theta\\
		{\rm s.t.}\ \  &\theta<UBD^k\\
		\ \    &f(x_j, y_j)+\nabla f(x_j, y_j)^T\begin{pmatrix}x-x_j\\y-y_j\end{pmatrix}\leq \theta\ \ \forall y_j\in T^k,\ \  \\
		\ \ & g_i(x_j, y_j)+\nabla g_i(x_j, y_j)^T\begin{pmatrix}x-x_j\\y-y_j\end{pmatrix}\leq 0,\ \forall i, \ \forall y_j\in T^k,\\ 
		\ \ & g_i(x_l, y_l)+\nabla g_i(x_l, y_l)^T\begin{pmatrix}x-x_l\\y-y_l\end{pmatrix}\leq 0,\ \forall i, \ \forall y_l\in S^k,\\ 
		\ \ &  x\in X, y\in Y\cap \mathbb{Z}^p, \theta\in\mathbb{R}.
	\end{aligned}
	\right.
\end{equation}

The classic OA algorithm for convex MINLP problem of \eqref{3.1-251030} is presented as follows:

\begin{breakablealgorithm}
	\caption{(OA Algorithm  for Convex MINLP)}
	\begin{algorithmic}[1]
		\STATE Initialization. Given an initial $y_0\in Y$, set $T^0=S^0:=\emptyset$, $UBD^0:=\infty$ and let $k:=1$
		\FOR{$k=1,2, \cdots,$}
		\IF {$y_k\in T$}  
		\STATE Solve subproblem $NLP(y_k)$ and denote the solution by $x_k$
		
		
		Set $T^k:=T^{k-1}\cup\{y_k\}$, $S^k:=S^{k-1}$ and $UBD^k:=\min\{UBD^{k-1}, f(x_k,y_k)\}$
		
		\ELSE
		\STATE Solve $F(y_k)$ and denote the solution by $x_k$
		
		
		Set $S^k:=S^{k-1}\cup\{y_k\}$, $T^k:=T^{k-1}$ and $UBD^k:=UBD^{k-1}$
		\ENDIF
		
		\STATE Solve $MP^k$ and obtain a new integer $y_{k+1}$
		
		Set $k:=k+1$ and go back to line 3

		\ENDFOR
	\end{algorithmic}
\end{breakablealgorithm}

Algorithms 2 and 3 differ only in the introduction of the new parameter  $\rho_j$. This seemingly minor addition, however, leads to a fundamentally different approach in several key respects. Unlike the classical OA algorithm, which relies solely on first-order Taylor expansions, our method constructs valid linear cuts that reformulate the problem into an equivalent MILP. The parameter  $\rho_j$ enables the generation of diverse MILP relaxations, producing a tighter polyhedral approximation of the original nonlinear feasible region. By explicitly leveraging the structural information of the constraint functions $g_i$, our algorithm adapts to the specific problem, moving beyond a generic, first-order linearization.

\begin{figure}[H] 
	\textbf{Example 3.1} Consider the following convex MINLP problem:
	
	\centering
	\begin{minipage}{\textwidth}
		\begin{equation}\label{3.36-251216}
\left\{ \begin{aligned}
				\min_{x,y} \quad & \frac{x^2}{10} - \frac{y}{4.5} + 2 + 0.001y^2 \\
				\text{s.t.} \quad & \frac{x^2}{20} + y \leq 20, \\
				& \frac{(x-1)^2}{40} - y \leq -4, \\
				& 0.275y^{1.5} - 10(x+0.1)^{0.5} \leq 0, \\
				& 0 \leq x \leq 20, \\
				& 0 \leq y \leq 20, y\in \mathbb{Z}.
			\end{aligned}
			\right.			
		\end{equation}
	\end{minipage}
\vspace{0.5\baselineskip}
\begin{minipage}{1\textwidth}
	The optimization process, starting from the initial point $(x_0, y_0) = (1, 4)$, yields the optimal solution $(x^*, y^*) = (1.9752, 14)$, with an optimal objective value of $f^* = -0.5249$.
\end{minipage}

\includegraphics[width=0.5\textwidth]{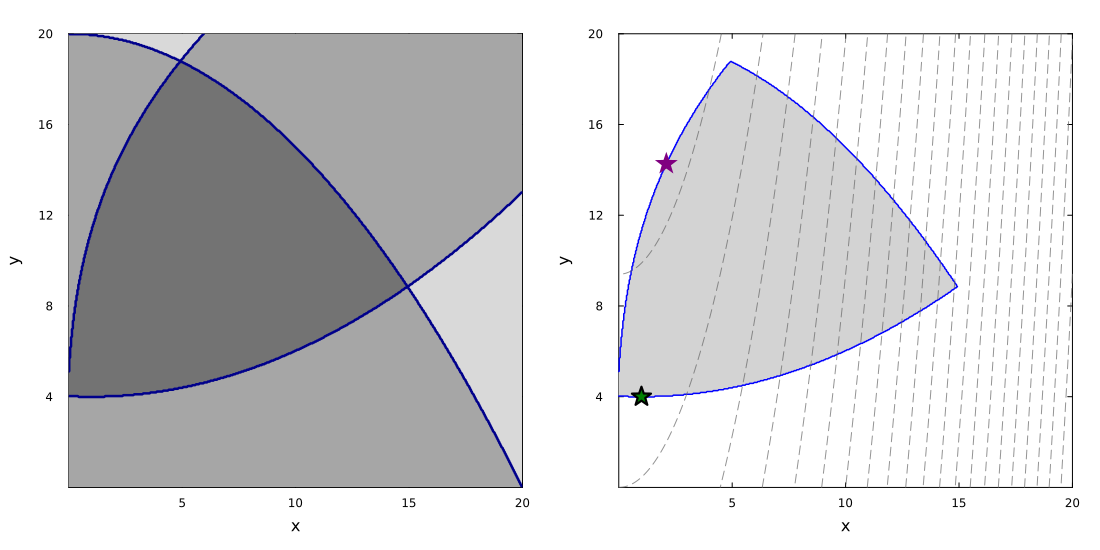}

\caption{ The figure to the left shows the feasible regions defined by the individual constraints of problem \eqref{3.36-251216}. The right figure shows the integer relaxed feasible region (i.e., the region without considering integer restrictions), contours of the objective function, the optimal solution (\textcolor{violet}{$\bigstar$}) of the problem, and the initial point(\textcolor[HTML]{008000}{$\bigstar$}).}
\label{fig:1}
\end{figure}

	\begin{figure}[H] 
	\centering
	\includegraphics[width=0.8\textwidth]{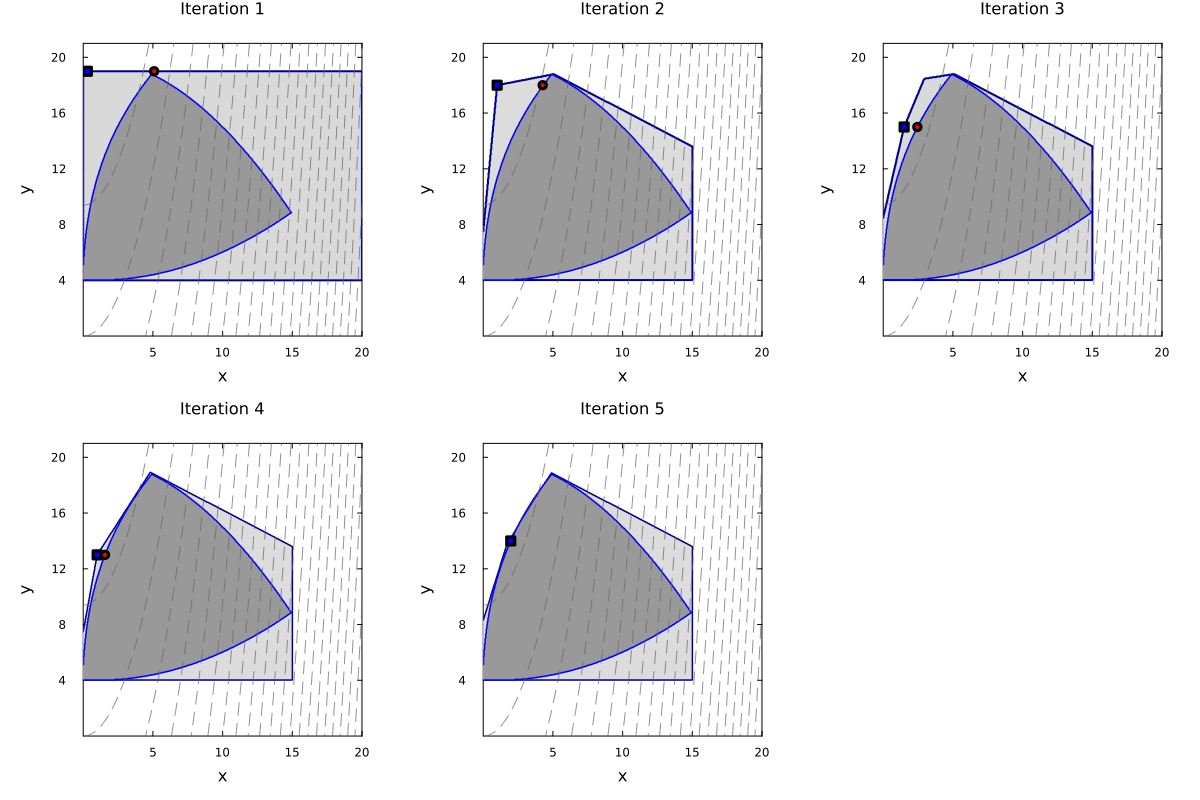}
	
	\caption{Applying the OA algorithm to problem \eqref{3.36-251216} results in multiple iterations, and the progress of the algorithm is shown, with each figure being an iteration. The figures show the feasible region defined by the nonlinear constraints in dark gray, and the light gray region shows the outer approximation obtained by the generated cuts. The squared dots (\textcolor{blue}{$\blacksquare$}) represent the solution obtained from the MILP master problem, and the circular dots (\textcolor{red}{$\bullet$}) represent the solution obtained by the NLP subproblem.}
	\label{fig:2}
\end{figure}

	\begin{figure}[H]
	\centering
	\includegraphics[width=0.8\textwidth]{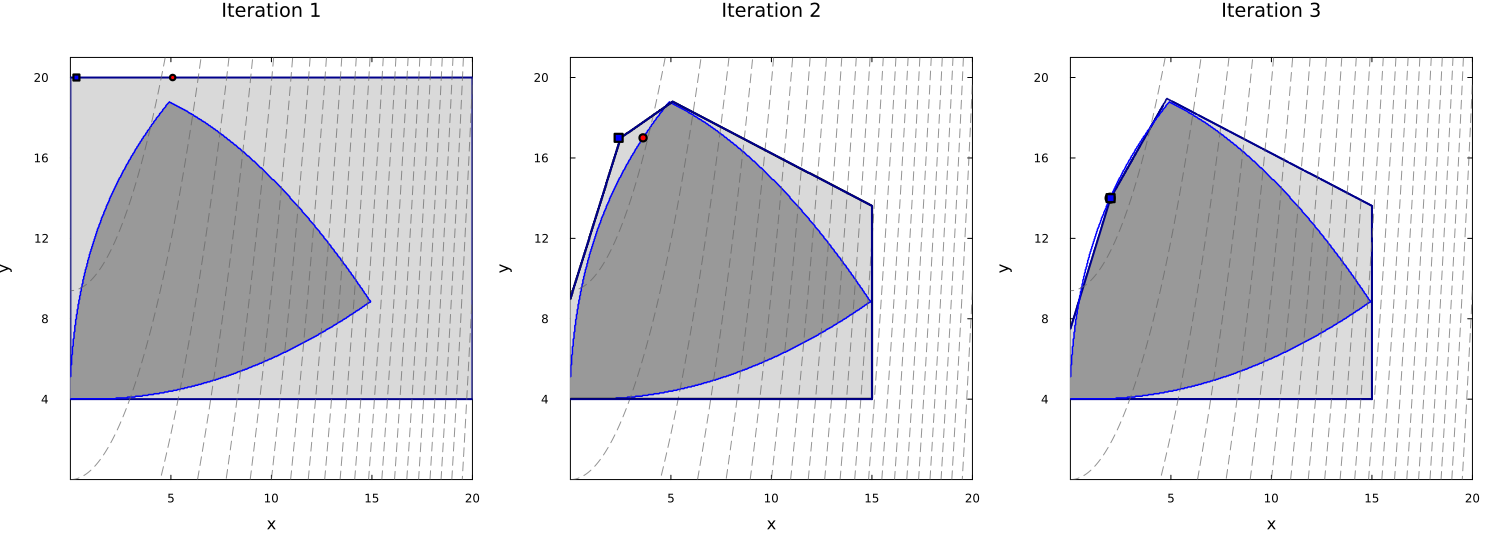}
	\caption{The figure demonstrates the iterative process of solving problem \eqref{3.36-251216} using the New OA algorithm, and it adopts the format of Fig~\ref{fig:2}.}
	\label{fig:3}
\end{figure}


Comparing Fig.~\ref{fig:2} and Fig.~\ref{fig:3} reveals that the New OA algorithm requires fewer iterations than the original OA algorithm to solve problem~\eqref{3.36-251216}. This improvement in efficiency can be attributed to the second iteration, illustrated in Fig.~\ref{fig:4}. In this step, the cutting plane generated by the New OA algorithm removes two feasible but non-optimal points within the master problem's feasible region- specifically,  $y=15$ and  $y=13$. By eliminating these points, the gap between the upper bound (UB) and lower bound (LB) is narrowed, which strengthens the cut and enhances the overall convergence rate of the algorithm.

\begin{figure}[H]
	\centering
	\includegraphics[width=0.8\textwidth]{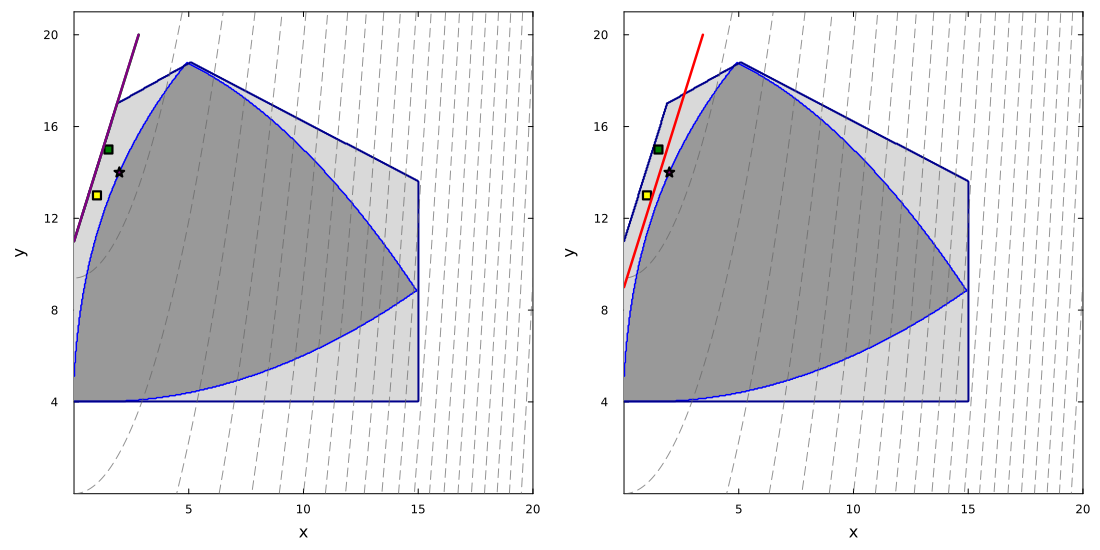}
	\caption{The left figure shows a cutting line (\textcolor{violet}{\rule{0.4cm}{0.6pt}}) generated by the OA algorithm during the second iteration. The right figure shows a cutting line (\textcolor{red}{\rule{0.4cm}{0.6pt}}) produced by the new OA algorithm during the second iteration, where $\rho_2 = 2.3464 / 12.8289 \approx 0.1829$. The yellow square points(\textcolor{yellow}{$\blacksquare$}) in the figure represent feasible points when $y = 13$, and the green square points(\textcolor[HTML]{008000}{$\blacksquare$}) represent feasible points when $y = 15$. The purple star(\textcolor{violet}{$\bigstar$}) represents the optimal solution to this problem. The format from Figure~2 is used here.}
	\label{fig:4}
\end{figure}

\section{Numerical results}\label{sec4}

This section presents a performance comparison between two investigated algorithms: the OA algorithm and the new OA algorithm. Thirty test instances sourced from MINLPLib \cite{BussiecDrudMeeraus2003} were employed for this comparison. The selected instances all satisfy the following criteria: they are convex optimization problems containing both integer and continuous variables, have a globally optimal objective value, include at least one nonlinear term in the objective function or constraints, and vary in their level of solving difficulty. All computational experiments were carried out under identical hardware conditions, specifically using an Intel Core i5-12500H processor (operating at 2.50 GHz) and 16.0 GB of RAM. For termination, and consistent with standard MINLP practices, we employed both an absolute optimality tolerance ($\epsilon$) and a relative optimality tolerance ($\epsilon_{\text{rel}}$). The search process will be terminated when either of the following conditions is met:$$f(\bar{x}, \bar{y}) - LB \leq \epsilon \quad \text{or} \quad \frac{f(\bar{x}, \bar{y}) - LB}{|f(\bar{x}, \bar{y})| + 10^{-10}} \leq \epsilon_{\text{rel}}$$Here, $LB$ denotes the current lower bound, $UB$ the current upper bound, and $\bar{x}, \bar{y}$ the best feasible solution identified during the search.  Numerically, these tolerances were set to $\epsilon = 10^{-5}$ and $\epsilon_{\text{rel}} = 10^{-3}$, respectively. Additionally, a maximum iteration limit of 900 was set for each problem instance to constrain the overall algorithm execution.The detailed description of the number of continuous variables, discrete variables, iteration counts, and run times for each problem is provided in Table~\ref{tab:detailed_examples}. For the instances that failed to run successfully, both the iteration count and run time are denoted by the symbol `---'. One instance in the test set, which the OA algorithm failed to solve within a $0.1\%$ optimality gap, was successfully solved by the new OA algorithm. Furthermore, the comparative results show that the new OA algorithm required fewer iterations than the OA algorithm for 16 instances. Even for those instances where the iteration count was comparable to or slightly higher, the new OA algorithm exhibited superior time efficiency, fully validating its effectiveness in enhancing overall solution efficiency.To further highlight the advantages of the new OA algorithm over the OA algorithm, the following section provides a detailed analysis based on two specific numerical examples.

For the instance \texttt{cvxnonsep\_normcon30}, a significant improvement in iteration efficiency can be observed: the number of iterations was reduced from 551 in the OA algorithm to 98, and the solution time decreased by 49.5\%, effectively lowering computational costs.

\begin{figure}[H]
	\centering
	\includegraphics[width=0.7\textwidth]{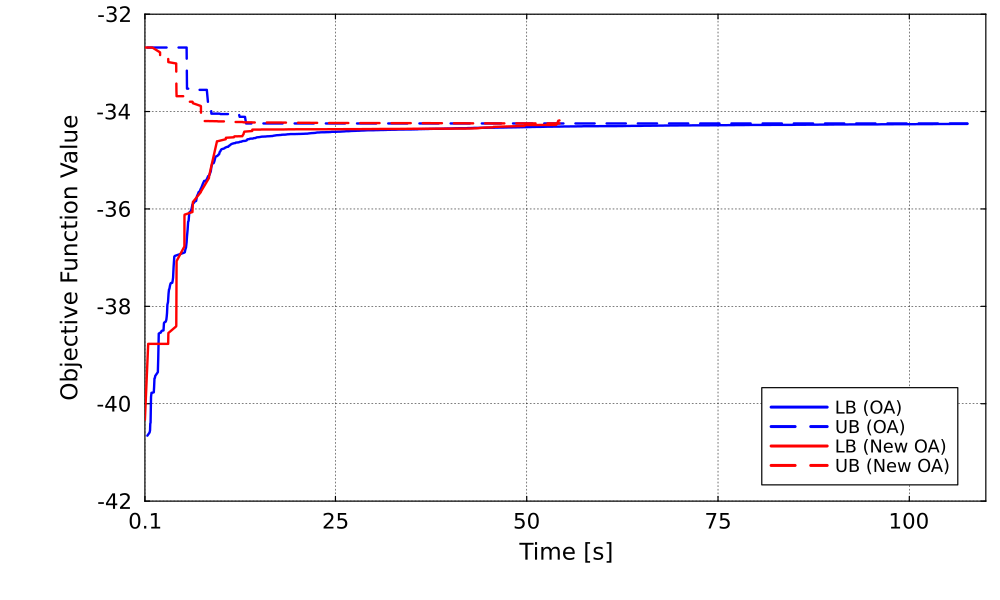}
	\caption{Bound profiles of the upper bound (UB) and lower bound (LB) over computation time for instance cvxnonsep\_normcon30, solved by the OA and New OA optimization algorithms.}
	\label{fig:5}
\end{figure}

For the instance \texttt{cvxnonsep\_psig20}, the new OA algorithm also demonstrated outstanding performance, requiring only 50 iterations and 3.90 seconds to achieve a 0.1\% optimality gap. Compared to the OA algorithm (128 iterations, 7.95 seconds), the new OA algorithm reduced the iteration count by 78 and shortened the solution time by 4.05 seconds.

\begin{figure}[H]
	\centering
	\includegraphics[width=0.7\textwidth]{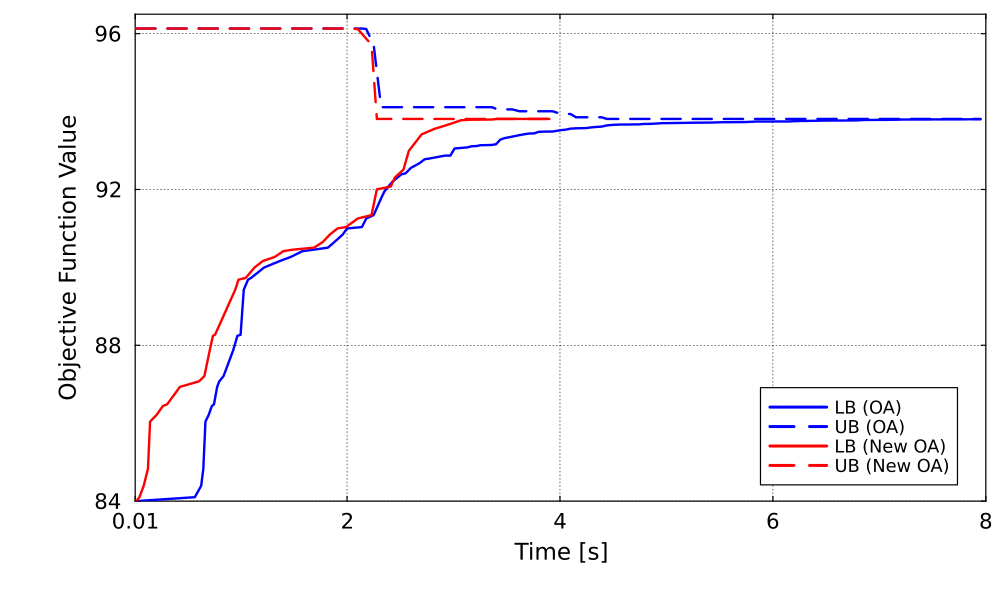}
	\caption{Bound profiles of the upper bound (UB) and lower bound (LB) over computation time for instance cvxnonsep\_psig20, solved by the OA and New OA optimization algorithms.}
	\label{fig:6}
\end{figure}

\begin{center}
	\setlength{\LTpre}{0pt}
	\setlength{\LTpost}{0pt}
	
	\begin{longtable}{
			>{\centering\arraybackslash}p{0.26\textwidth}
			>{\centering\arraybackslash}p{0.15\textwidth}
			>{\centering\arraybackslash}p{0.18\textwidth}
			>{\centering\arraybackslash}p{0.12\textwidth}
			>{\centering\arraybackslash}p{0.10\textwidth}
		}
		
		\caption{Comparison of the OA and New OA Algorithms.} \\[-18pt]
		\label{tab:detailed_examples} \\

		\toprule
		\makecell[c]{Instance(\# discrete \\ vars, \# continuous \\ vars, \# constraints)}
		& \makecell[c]{Solution \\algorithm}
		& \makecell[c]{Optimization \\ Value}
		& Iterations
		& Time(s) \\
		\midrule
		\endfirsthead

		\caption{Comparison of the OA and New OA Algorithms.} \\
		\toprule
		\makecell[c]{Instance(\# discrete \\ vars, \# continuous \\ vars, \# constraints)}
		& \makecell[c]{Solution \\algorithm}
		& \makecell[c]{Optimization \\ Value}
		& Iterations
		& Time(s) \\
		\midrule
		\endhead
		
		\midrule
		\multicolumn{5}{r}{Continued on next page} \\
		\endfoot
		
		\bottomrule
		\endlastfoot
		\multirow{2}{*}{\makecell[c]{clay0203m\\(18, 12, 54)}} &
		OA & \multirow{2}{*}{41573.26} & 10 & 7.82 \\*  
		& New OA & & 8 & 6.46 \\
		\midrule
		
		\multirow{2}{*}{\makecell[c]{clay0204hfsg\\(32, 132, 234)}} &
		OA & \multirow{2}{*}{6545.00} & 2 & 2.80 \\*
		& New OA & & 2 & 0.42 \\
		\midrule
		
		\multirow{2}{*}{\makecell[c]{clay0204m\\(32, 20, 90)}} &
		OA & \multirow{2}{*}{6545.00} & 6 & 3.55 \\*
		& New OA & & 4 & 1.29 \\
		\midrule
		
		\multirow{2}{*}{\makecell[c]{clay0303m\\(21, 12, 66)}} &
		OA & \multirow{2}{*}{26669.10} & 16 & 7.23 \\*
		& New OA & & 16 & 3.79 \\
		\midrule
		
		\multirow{2}{*}{\makecell[c]{cvxnonsep\_normcon20\\(10, 10, 1)}} &
		OA & \multirow{2}{*}{-21.75} & 122 & 16.10 \\*
		& New OA & & 100 & 4.33 \\
		\midrule
		
		\multirow{2}{*}{\makecell[c]{cvxnonsep\_normcon30\\(15, 15, 1)}} &
		OA & \multirow{2}{*}{-34.24} & 551 & 107.59 \\*
		& New OA & & 98 & 54.28 \\
		\midrule
		
		\multirow{2}{*}{\makecell[c]{cvxnonsep\_normcon40\\(20, 20, 1)}} &
		OA & \multirow{2}{*}{-32.63} & 168 & 147.18 \\*
		& New OA & & 193 & 137.96 \\
		\midrule
		
		\multirow{2}{*}{\makecell[c]{cvxnonsep\_nsig20\\(10, 10, 1)}} &
		OA & \multirow{2}{*}{80.95} & 133 & 8.62 \\*
		& New OA & & 120 & 4.42 \\
		\midrule
		
		\multirow{2}{*}{\makecell[c]{cvxnonsep\_nsig30\\(15, 15, 1)}} &
		OA & \multirow{2}{*}{130.63} & 487 & 81.92 \\*  
		& New OA & & 491 & 58.03 \\
		\midrule
		
		\multirow{2}{*}{\makecell[c]{cvxnonsep\_nsig40\\(20, 20, 1)}} &
		OA & \multirow{2}{*}{133.96} & 824 & 351.27 \\*
		& New OA & & 389 & 104.68 \\
		\midrule
		
		\multirow{2}{*}{\makecell[c]{cvxnonsep\_pcon20\\(10, 10, 1)}} &
		OA & \multirow{2}{*}{-21.51} & 54 & 12.37 \\*
		& New OA & & 51 & 1.36 \\
		\midrule
		
		\multirow{2}{*}{\makecell[c]{cvxnonsep\_pcon30\\(15, 15, 1)}} &
		OA & \multirow{2}{*}{-35.99} & 140 & 17.71 \\*
		& New OA & & 50 & 4.72 \\
		\midrule
		
		\multirow{2}{*}{\makecell[c]{cvxnonsep\_pcon40\\(20,20,1)}} &
		OA & \multirow{2}{*}{-46.60} & 262 & 52.36 \\*
		& New OA & & 258 & 34.82 \\
		\midrule
		
		\multirow{2}{*}{\makecell[c]{cvxnonsep\_psig20\\(10,10,0)}} &
		OA & \multirow{2}{*}{93.81} & 128 & 7.95 \\*
		& New OA & & 50 & 3.90 \\
		\midrule
		
		\multirow{2}{*}{\makecell[c]{cvxnonsep\_psig30\\(15,15,0)}} &
		OA & \multirow{2}{*}{79.00} & 330 & 48.69 \\*
		& New OA & & 330 & 31.13 \\
		\midrule
		
		\multirow{2}{*}{\makecell[c]{ex4\\(25,11,30)}} &
		OA & \multirow{2}{*}{-8.06} & 3 & 0.45 \\*
		& New OA & & 3 & 0.19 \\
		\midrule
		
		\multirow{2}{*}{\makecell[c]{ex1223b\\(4,3,9)}} &
		OA & \multirow{2}{*}{4.58} & 5 & 3.62 \\*
		& New OA & & 5 & 0.96 \\
		\midrule
		
		\multirow{2}{*}{\makecell[c]{ex1223\\(4,7,13)}} &
		OA & \multirow{2}{*}{4.58} & 6 & 2.14 \\*
		& New OA & & 5 & 0.16 \\
		\midrule
		
		\multirow{2}{*}{\makecell[c]{fac1\\(6,16,18)}} &
		OA & \multirow{2}{*}{160912612.40} & 4 & 1.75 \\*
		& New OA & & 4 & 0.99 \\
		\midrule
		
		\multirow{2}{*}{\makecell[c]{fac2\\(12,54,33)}} &
		OA & \multirow{2}{*}{331837498.20} & -- & -- \\*
		& New OA & & 7 & 1.35 \\
		\midrule
		
		\multirow{2}{*}{\makecell[c]{flay02h\\(4,42,51)}} &
		OA & \multirow{2}{*}{37.95} & 2 & 0.91 \\*
		& New OA & & 2 & 0.86 \\
		\midrule
		
		\multirow{2}{*}{\makecell[c]{flay02m\\(4,10,11)}} &
		OA & \multirow{2}{*}{37.95} & 2 & 0.69 \\*
		& New OA & & 1 & 0.06 \\
		\midrule
		
		\multirow{2}{*}{\makecell[c]{flay03h\\(12,110,144)}} &
		OA & \multirow{2}{*}{48.99} & 8 & 3.61 \\*
		& New OA & & 9 & 0.83 \\
		\midrule
		
		\multirow{2}{*}{\makecell[c]{flay03m\\(12,14,24)}} &
		OA & \multirow{2}{*}{48.99} & 8 & 1.29 \\*
		& New OA & & 7 & 0.97 \\
		\midrule
		
		\multirow{2}{*}{\makecell[c]{flay04m\\(24,18,42)}} &
		OA & \multirow{2}{*}{54.41} & 29 & 7.93 \\*
		& New OA & & 21 & 3.96 \\
		\midrule
		
		\multirow{2}{*}{\makecell[c]{syn30m\\(30, 70, 167)}} &
		OA & \multirow{2}{*}{138.16} & 7 & 5.25 \\*
		& New OA & & 4 & 2.88 \\
		\midrule
		
		\multirow{2}{*}{\makecell[c]{synthes2\\(5, 6, 14)}} &
		OA & \multirow{2}{*}{73.04} & 3 & 0.36 \\*
		& New OA & & 3 & 0.19 \\
		\midrule
		
		\multirow{2}{*}{\makecell[c]{synthes3\\(8, 9, 23)}} &
		OA & \multirow{2}{*}{68.01} & 4 & 1.56 \\*
		& New OA & & 4 & 0.61\\
		\midrule
		
		\multirow{2}{*}{\makecell[c]{sssd18-06\\(126, 24, 66)}} &
		OA & \multirow{2}{*}{397992.29} & 16 & 16.51 \\*
		& New OA & & 17 & 13.26 \\
		\midrule
		
		\multirow{2}{*}{\makecell[c]{sssd12-05\\(75,20,52)}} &
		OA & \multirow{2}{*}{281408.64} & 11 & 7.50 \\*
		& New OA & & 10 & 2.72 \\
		
	\end{longtable}
\end{center}

Therefore, based on the boundary convergence trends and the experimental data in Table \ref{tab:detailed_examples}, it is evident that New OA significantly improves the performance of the original OA algorithm in solving MINLP problems. Specifically, for test cases \texttt{cvxnonsep\_normcon40}, \texttt{cvxnonsep\_nsig30}, \texttt{flay03h}, and \texttt{sssd18-06}, the new OA algorithm achieves a notable reduction in computation time despite requiring a higher number of iterations. This demonstrates that, while maintaining the same solution tolerance, the enhanced algorithm attains greater computational efficiency, completing the solution process more quickly. These results confirm the practical effectiveness of the new OA algorithm.


\section{Concluding remarks}


This paper presents a novel OA method for convex  MINLP problems. The core of the approach is the reformulation of the original MINLP as an equivalent MILP master problem and the MILP relaxation can generate a tighter polyhedral region to approximate the nonlinear feasible region when compared  with the conventional OA.

A key innovation is the introduction of new parameters $\rho_j$ in \eqref{3.24-251030} derived from feasible subproblems. By strategically incorporating these parameters into the MILP relaxations, the new OA algorithm is constructed to produce MILP relaxations that are able to generate tighter linear approximations than the classical OA. Numerical experiments demonstrate that the proposed new OA method may be more efficient than the classical OA, as it generates relaxations that more closely approximate the original nonlinear feasible region.

This efficiency stems from an alternative mechanism for generating valid linear cuts, which extends beyond the reliance on standard first-order Taylor expansions. The use of problem-specific parameters highlights that effectively leveraging the inherent structure of the MINLP is also crucial for the performance of the OA algorithm.





\end{document}